\documentclass[11pt]{amsart}

\usepackage{amsmath,amsthm, amscd, amssymb, amsfonts, enumerate}
\usepackage{marginnote}

\DeclareMathOperator{\Ad}{Ad}
\DeclareMathOperator{\ad}{ad}
\DeclareMathOperator{\End}{End}
\DeclareMathOperator{\grad}{grad}

\renewcommand{\Im}{\operatorname{Im}}
\DeclareMathOperator{\En}{E}

\newcommand{\Heis}{\mathrm{H}}

\newcommand{\RR}{\mathbb R}

\newcommand{\CC}{\mathbb C}
\newcommand{\la}{\langle}
\newcommand{\ra}{\rangle}

\newcommand{\af}{\mathfrak a}
\newcommand{\bb}{\mathfrak b}
\newcommand{\ggo}{\mathfrak g}
\newcommand{\hh}{\mathfrak h}
\newcommand{\kk}{\mathfrak k}
\newcommand{\mm}{\mathfrak m}
\newcommand{\nn}{\mathfrak n}
\newcommand{\sso}{\mathfrak{so}}
\newcommand{\vv}{\mathfrak v}
\newcommand{\zz}{\mathfrak z}

\theoremstyle{plain}
\newtheorem{theorem}{Theorem}[section]
\newtheorem{cor}[theorem]{Corollary}
\newtheorem{proposition}[theorem]{Proposition}
\newtheorem{lem}[theorem]{Lemma}

\newtheorem*{utheorem}{Theorem}

\theoremstyle{definition}
\newtheorem{defn}[theorem]{Definition}

\theoremstyle{remark}
\newtheorem{remark}[theorem]{Remark}

\numberwithin{equation}{section}

\begin{document}


\title[The geodesic flow on  nilmanifolds]{The geodesic flow on  nilmanifolds}

\author{Alejandro Kocsard}
\address{A. Kocsard. IME - Universidade Federal Fluminense. Rua M\'rio Santos Braga S/
N, 24020-140
Niter\'oi, RJ, Brazil}
\email{akocsard@id.uff.br}

\author{Gabriela P. Ovando}
\address{G. Ovando. CONICET - Universidad Nacional de Rosario. Depto. de Matemática, ECEN - FCEIA, Pellegrini 250. 2000 Rosario}
\email{gabriela@fceia.unr.edu.ar}

\author{Silvio Reggiani}
\address{S. Reggiani. CONICET - Universidad Nacional de Rosario. Depto. de Matemática, ECEN - FCEIA, Pellegrini 250. 2000 Rosario}
\email{reggiani@fceia.unr.edu.ar}

\thanks{Partially supported by ANPCyT and SCyT - Universidad Nacional de Rosario}
\begin{abstract} In this paper we study the geodesic flow on nilmanifolds equipped with a left-invariant metric. We write the underlying definitions and find general formulas for the Poisson involution. As an example we develop the Heisenberg Lie group equipped with its canonical metric. We prove that a family of first integrals giving the complete integrability can be read off at the Lie algebra of the isometry group. We also explain the complete integrability on compact quotients and for any invariant metric.
\end{abstract}

\maketitle

\section{Introduction}

Given a smooth manifold $M$, any complete Riemannian structure
$\langle\cdot,\cdot\rangle$ induces the geodesic flow
$\Gamma\colon M\times\mathbb{R}\to M$ which can be defined as the
Hamiltonian flow associated to the energy function
$E(v):=1/2\langle v,v\rangle$ on $TM$.

Usually, this flow is not integrable in the sense of Liouville and it
is generally expected that the integrability of the geodesic flow
imposes important obstructions to the topology of the supporting
manifold.

However, contrasting some results of Taimanov~\cite{Ta1,Ta2} on
topological obstructions for real-analytic manifolds supporting
real-analytic integrable geodesic flows with some smooth examples of
smoothly integrable geodesic flows on manifolds that do not satisfy
the above obstructions constructed by Butler~\cite{Bu0}, and Bolsinov
and Taimanov~\cite{BT}, we observe that the regularity of first
integrals plays a fundamental role that nowadays is not completely
well understood.

For that reason, when dealing with locally homogeneous manifolds to have the
possibility of constructing real-analytic first integrals is very
desirable.

In the advances reached in the theory of Hamiltonian systems in the
1980's one can recognize the role of Lie theory in the study of
several examples. This is the case of the so known Adler-Kostant-Symes
\cite{Ad,Ko,Sy} scheme used for the study of some mechanical systems
and of the so known Thimm's method for the study of the geodesic flow
\cite{Th}. In both cases the main examples arise from semisimple Lie
groups. These results appeared parallel to the many studies given by
the Russian school which can be found for instance in \cite{F-T}. Some
non-homogeneous examples are found in \cite{P-S}.

For other Hamiltonian systems on non semisimple Lie groups only few examples and generalizations are known in the case of the geodesic flow. This is the situation for nilpotent and solvable Lie groups or even their compact quotients which are locally homogeneous manifolds. 

On the one hand Butler proved the integrability of the geodesic flow on Heisenberg-Reiter 2-step nilpotent Lie groups, among them one can find the Heisenberg Lie group  \cite{Bu1}. He also found an algebraic condition for non-integrability. On the other hand, Eberlein started a study of the geometry concerning the geodesic flow on 2-step Lie groups following his own and longer study in this topic, giving a good material and references in \cite{Eb1, Eb2}. This study of Eberlein is much more general and is mixed with many other geometrical questions. 

In the present paper we concentrate in the geodesic flow of Lie groups endowed with a metric invariant by left-translations. In the first part we write the basic definitions and get general conditions and formulas for the involution of first integrals making use of the Lie theory tools, that is, assuming some natural identifications which are not present in the work of Butler. However, all this stuff is good explained along the work. We write the underlying results on the tangent Lie groups and put special emphasis on 2-step nilpotent Lie groups, we take as nilmanifold after \cite{Wo},  for which there exists a developed geometrical theory and several examples and applications, see \cite{Eb}. 

We apply the results we get to the case of the Heisenberg Lie group $\Heis_n$ of dimension $2n+1$ equipped with the canonical left-invariant metric. Although this is a naturally reductive space  the methods of Thimm do not apply in this case. 

One of our motivations is to investigate the nature of the first integrals one can construct. We prove that all the first integrals we get can be visualized on the isometry group. In fact, this is the case of quadratic polynomials which are invariant, so as first integrals arising from Killing vector fields. Recall that, given a Killing vector field $X^*$ on a Riemannian manifold $M$, one has a first integral on the tangent Lie bundle $TM$ defined by $f_{X^*}(v)= \la X^*, v\ra$. We proved that
\begin{enumerate}[(i)]

\item There is a bijection   between the set of quadratic first integrals of the geodesic flow on $\Heis_n$ ---with the canonical metric--- and the Lie subalgebra of skew-symmetric derivations of the Heisenberg Lie algebra $\hh_n$, so that involution of quadratic first integrals would correspond to a torus of skew-symmetric derivations (Theorem 4.1). Actually, a general formulation of  quadratic polynomials on a 2-step nilpotent Lie algebras to be first integrals is found so as the pairwise commutativity condition. 

\item The linear morphism $X^* \to f_{X^*}$ builds a Lie algebra isomorphism onto its image.
\end{enumerate}

Item (ii) gives an answer to the question formulated in \cite{Th} and it is the first example we found of this situation.

Making use of all results we found a new family of first integrals for the geodesic flow on $\Heis_n$. 

After that, we consider an arbitrary lattice $\Lambda$ and, passing to the quotient, we explain the integrability of the geodesic flow on the compact space $\Lambda \backslash \Heis_n$. Finally, we consider any left-invariant metric on the Heisenberg Lie group $\Heis_n$ and prove that the corresponding geodesic flow is also completely integrable.

\section{The canonical symplectic structure}

Let $M$ denote a differentiable manifold. Its cotangent bundle $T^*M$ admits a canonical symplectic form $\tilde{\Omega}$ constructed as  follows: for every $\xi \in T_\eta(T^*M)$ define the 1-form $\tilde{\Theta}$ by
$$
\tilde{\Theta}(\xi) := \eta(d\tilde{\pi}(\xi)),
$$
where $\tilde{\pi}: T^*M \to M$ is the canonical projection which
assigns to $\eta\in T^*_pM$ the base point $p\in M$. Now the
symplectic form of $T^*M$ is given by
$$
\tilde{\Omega} := -d\tilde{\Theta}.
$$

If $M$ is a Riemannian manifold with metric tensor $\langle\cdot, \cdot\rangle$, one induces a symplectic form $\Omega$ in the tangent bundle $TM$ of $M$. In fact, take 
$$
\Theta(\xi) := \langle v, d\pi(\xi)\rangle,
$$
where $\xi \in T_v(TM)$ and $\pi\colon TM \to M$  denotes the canonical projection, and then  define
$$
\Omega := -d\Theta.
$$

It is easy to see that both $\Theta$ and $\Omega$ are the pull-back by the natural diffeomorphism $\phi: TM \to T^*M$ given by $\phi(v) = \langle v, \cdot\rangle$, that is $\Theta = \phi^*(\tilde{\Theta})$ and $\Omega = \phi^*(\tilde{\Omega})$.

We shall concentrate  the study in the case where $M$ is a Lie group endowed with a left-invariant metric. Let $N$ denote a Lie group endowed with a left-invariant metric $\langle\cdot, \cdot\rangle$. Let $\mathfrak n$ be the Lie algebra of $N$. The tangent bundle $TN$ of $N$ is parallelizable and it is identified with
$$
TN \simeq N \times \mathfrak n  = \{(p, Y): p \in N,\, Y \in \mathfrak n\}.
$$

With this identification, we consider the pair $(p, Y)$ as the tangent
vector $dL_p(Y_e) \in T_pN$. Recall that $TN\simeq N\times \nn$ is a
Lie group regarded as the direct product of $N$ and the abelian group
$\nn$ and so
$$
T_{(p, Y)}(TN) \simeq \mathfrak n \times \mathfrak n = \{(U, V): U, V \in \mathfrak n\}.
$$

When the base point  is not clear from the context, we will use the notation $(U, V)_{(p, Y)}$ for a tangent vector in $T_{(p, Y)}(TN)$, $p\in N, Y\in \nn$. The pair $(U, V)$ denotes  a left-invariant field on $TN$ (i.e., the left-invariant vector field $U$ on $N$ in the first coordinate, and the constant field $V$ in $\mathfrak n$ in the second one).

With this identification, and since the metric on $N$ is
left-invariant, the canonical $1$-form on $TN$ has the following form 
$$
\Theta_{(p, Y)}(U, V) = \langle Y, U\rangle.
$$

Recall that the curve on $TN$ starting at $(p,Y)$ 
$$
c(t) = (p\exp(tU), tV + Y)
$$ 
has initial speed $(U, V) \in T_{(p, Y)}(TN)$. So, 
$$
(U, V)_{(p, Y)}\Theta(U', V') = \frac{d}{dt}\bigg|_0 \Theta((U', V')_{c(t)}) = \frac{d}{dt}\bigg|_0 \langle tV + Y, U'\rangle,
$$
and thus 
$$
(U, V)_{(p, Y)}\Theta(U', V') = \langle V, U'\rangle.
$$
Finally, 
\begin{align*}
\Omega_{(p,Y)}((U, V), (U', V')) & = -d\Theta((U, V), (U', V')) \\
& = -\{(U, V)_{(p, Y)}\Theta(U', V') - (U', V')_{(p, Y)}\Theta(U, V) \\
& \hspace{12pc} - \Theta_{(p, Y)}([(U, V), (U', V')])\} \\
& = -\langle V, U'\rangle + \langle U, V'\rangle + \Theta_{(p, Y)}([U, U'], 0),
\end{align*}
and therefore
\begin{equation}\label{symplectic-tan}
\Omega_{(p, Y)}((U, V), (U', V')) =  \langle U, V'\rangle -\langle V, U'\rangle + \langle Y, [U, U']\rangle.
\end{equation}

\begin{remark} A similar work can be done in the cotangent bundle
$$
T^*N \simeq N \times \mathfrak n^* = \{(h, \alpha): h \in N,\, \alpha \in \mathfrak n^*\}.
$$
Here the pair $(h, \alpha)$ is identified with $\alpha \circ dL_{h^{-1}}|_h$, after identifying $\mathfrak n \simeq T_eN$ as usual. Notice that if we move $h$ but fix  $\alpha$  we get  a left-invariant $1$-form on $N$. Using similar notations and identifications as in the case of the
tangent bundle, one can get:
$$
T_{(h, \alpha)}(T^*N) = \mathfrak n \times \mathfrak n^* = \{(U, \beta): U \in \mathfrak n,\, \beta \in \mathfrak n^*\},
$$
and hence, 
$$
\tilde \Theta_{(h, \alpha)}(U, \beta) = \alpha_e(\delta L_{h^{-1}}U_h).
$$

A curve in $T^*N$ with initial speed $(U, \beta) \in T_{(h, \alpha)}(T^*N)$ is given by
$$
c(t) = (h\exp(tU), t\beta + \alpha).
$$

Analogously to the tangent bundle one can show that
$$
(U, \beta)_{(h, \alpha)}\Theta(U', \beta') = \beta(U'),
$$
and
$$
\tilde\Omega_{(h, \alpha)}((U, \beta), (U', \beta')) = -\beta(U') + \beta'(U) + \alpha([U, U']).
$$
\end{remark}

 The Lie group $N$ is equipped with the product metric of $N\times  \nn$ which  is invariant under the left-translations $L_g$ for $g\in N$. In other words the Lie group $N$ acts on $TN$ by isometries.  From now we shall assume that $TN$ is endowed with this metric. 

 The next definitions introduce the Hamiltonian vector field and the gradient field.

\begin{defn} Let $f: TN \to \mathbb R$ be a smooth function. 
\begin{itemize}
\item The Hamiltonian vector field $X_f$ of $f$ is defined by
\begin{equation}\label{hamvec}
df_{(p, Y)}(W) = \Omega_{(p, Y)}(X_f, W) \qquad \mbox{ for } W\in T_p N,
\end{equation}
with Hamiltonian equation
\begin{equation}\label{hameq}
c'(t)=X_f(c(t)) \qquad \mbox{ for a smooth curve } c: \RR \to TN.
\end{equation}
\item The gradient field for $f$, denoted $\grad f$, is the vector field on $TN$ given by
\begin{equation}\label{gradvec}
df_{(p, Y)}(W) = \la \grad_{(p,Y)} f, W\ra \qquad \mbox{ for } W\in T_p N.
\end{equation}

\end{itemize}
\end{defn}

Let $\grad_{(p, Y)}f = (U, V)$, for $U, V\in \nn$ denote the gradient field of a smooth function  $f: TN \to \mathbb R$. We shall compute the Hamiltonian vector field $X_f$ in terms of the gradient vector field. 

Let us denote $X_f(p, Y) = (U'', V'')$. If $(U', V')$ is an arbitrary element of $T_{(p, Y)}TN$ we have that 
\begin{align*}
  \langle U, U'\rangle + \langle V, V'\rangle & = \langle\grad_{(p, Y)}f, (U', V')\rangle \\
  & = df_{(p, Y)}(U', V') \\
  & = \Omega_{(p, Y)}((U'', V''), (U', V')) \\
	& =  \langle U'', V'\rangle -\langle V'', U'\rangle  + \langle Y, [U'', U']\rangle.
\end{align*}
Now choose  $U' = 0$, so that the equalities above reduce to
$$\langle V, V'\rangle =  \langle U'', V'\rangle$$
implying  that $U'' = V$. By making use of this  in  the relations above one gets
\begin{equation}\label{eq:V''}
\langle U, U'\rangle = -\langle V'', U'\rangle + \langle Y, [V, U']\rangle \qquad \mbox{ for all } U'\in \nn. 
\end{equation}

Hence the Hamiltonian vector field for $f: TN \to \RR$  is given by
\begin{equation}\label{hamvf}
 X_f(p,Y)=(V, -U + \ad^t(V)Y) \qquad \mbox{ for }\quad \grad_{(p,Y)}f=(U,V)
\end{equation}
where $\ad^t(V)$ denotes the transpose of $\ad(V)$ relative to the metric on $\nn$.

Recall that a symplectic structure on a symplectic manifold $(N, \Omega)$ defines a {\em Poisson bracket} on $C^{\infty}(M)$ in the following way. Let $f, g: TN \to \RR$ be smooth functions, then its Poisson bracket is given by
\begin{equation}\label{poisson}
\{f, g\}{(p,Y)}=\Omega_{(p,Y)}(X_f, X_g)
\end{equation}

and one says that $f$ and $g$ are {\em in involution} if $\{f, g\}=0$, alternatively $f, g$ Poisson commute.

\begin{lem} \label{lema1} Let $(N, \la \,,\, \ra)$ denote a Lie group equipped with a left-invariant metric and let $f:TN \to \RR$ denote a smooth function. If one denotes the gradient of $f$ by  $\grad_{(p,Y)}f =(U, V)$, then 
\begin{enumerate}[(i)]
\item the Hamiltonian vector field of $f$ is given by
\begin{equation}\label{hamiltonianvf}
X_f(p, Y)=(V, \ad^t(V)(Y)-U)
\end{equation}
where $\ad^t(V)$ denotes the transpose of $\ad(V)$ with respect to the given left-invariant metric on $N$.

\item The Poisson bracket of a pair of functions $f,g:TN \to \RR$ is given by
\begin{equation}\label{poissonfg}
\begin{array}{rcl}
\{f,g\}(p,Y) & = & -\Omega_{(p,Y)}(\mathrm A \grad f, \mathrm A \grad g) \\
& = & \la V',U\ra - \la V, U'\ra - \la Y, [V,V']\ra, 
\end{array}
\end{equation}
where $\mathrm A (U,V)=(V,U)$.
\end{enumerate}
\end{lem}
\begin{proof} The statement $(i)$  was proved above. For the proof of the second statement assume $\grad_{(p,Y)} f=(U,V)$ and  $\grad_{(p,Y)} g=(U',V')$. Then 
$$
\begin{array}{rcl}
\{f,g\}(p,Y) & = & \Omega_{(p,Y)}(X_f, X_g) \\
& = &  \Omega_{(p,Y)}(V, \ad^t(V)(Y)-U), (V', \ad^t(V')(Y)-U')\\
& = & \la V',U\ra - \la V, U'\ra + \la Y, [V',V]\ra \\
& = & -\left( \la V, U'\ra - \la V',U\ra  + \la Y, [V,V']\ra \right)
\end{array}
$$
which concludes the proof.
\end{proof}

\subsection{Example: 2-step nilpotent Lie groups}

 Assume $N$ is a connected 2-step nilpotent  Lie group with Lie algebra  $\nn$. Whenever $N$ is furnished with a left-invariant metric $\la\,,\,\ra$, one can read several geometrical features of $N$ at the Lie algebraic level. In fact, the Riemannian metric is determined at the identity element so that the Lie algebra $\nn$ can be decomposed into a orthogonal direct sum
$$
\nn=\vv \oplus \zz \qquad \quad\mbox{ with  } \vv=\zz^{\perp}
$$
where as usual  $\zz$ denotes the center of $\nn$. The
Lie bracket on $\nn$ induces for $Z\in \zz$ the skew-symmetric  linear map
$j(Z):\vv \to \vv$  given by
\begin{equation}\label{br}
\begin{array}{rcl}
\la [U,V], Z\ra & = & \la j(Z) U,V\ra \qquad \mbox{ for } Z\in \zz, U,V\in \vv.
\end{array}
\end{equation}

Conversely, let $(\af, \la \,,\,\ra_{\af})$ and $(\bb, \la \,,\,\ra_{\bb})$ denote
vector spaces endowed with respective inner products. Let $\nn = \bb \oplus \af$ denote the
 direct sum as vector spaces 
and let $\la\,,\,\ra$ denote the product metric on $\nn$ given by
\begin{equation}\label{met}
\la \,,\,\ra_{|_{\af \times \af}}=\la \,,\,\ra_{\af}\qquad
\la \,,\,\ra_{|_{\bb \times \bb}}=\la \,,\,\ra_{\bb} \qquad
\la \af, \bb \ra=0.
\end{equation}

Let $j:\af \to \End(\bb)$ be a linear map such that $j(Z)$ is skew-symmetric with respect to $\la \,,\,\ra_{\bb}$ for every $Z\in \af$.  Then $\nn$
 becomes a 2-step nilpotent Lie algebra if one defines a Lie bracket by the relation in (\ref{br}) and so that $\af$ is contained in the center of $\nn$,  $\af \subseteq \zz$; actually 
$$\begin{array}{rcl}
\zz & = & \af \oplus \{ V \in \bb : [V, U]=0 \mbox{ for all } U\in \bb\} \\
& = & \af \oplus \cap_{Z\in\af} \ker j(Z).
\end{array}
$$
By using the translations on the left, the corresponding connected Lie group $N$ is endowed with a Riemannian metric.

\medspace

The inner product $\la\,,\,\ra$ produces  a decomposition of the center of the Lie algebra $\nn$   as a  orthogonal direct sum as vector spaces
$$\zz=\ker j \oplus C(\nn),$$
where $C(\nn)$ denotes the commutator of $\nn$ and the linear map $j$ is injective if and only if there is no Euclidean factor in the
De Rahm decomposition  of the simply connected Lie group $(N, \la\,,\,\ra)$ (see
\cite{Go}).

Let $\nn$ be a 2-step nilpotent Lie algebra. We say that $\nn$ is {\em non-singular} if $\ad(X) : \nn \to \zz$  is surjective for all $X \in \nn-\zz$. One can show that 
the following properties are
equivalent - see for instance \cite{Eb}:
\begin{enumerate}[(i)]
\item  $\nn$ is non-singular;
\item for every inner product $\langle\cdot, \cdot\rangle$ on $\nn$ and every nonzero element $Z$ of $\zz$ the
linear map $j(Z)$ is non-singular;
\item for some inner product $\la \,,\,\ra$ on $\nn$ and every non-zero element $Z$ of $\zz$ the
linear map $j(Z)$ is non-singular.
\end{enumerate}

Let us denote by $W_{\mathfrak v}$ (resp.\ $W_{\mathfrak z}$) the $\mathfrak v$-component (resp.\ $\mathfrak z$-component) of an element $W \in \mathfrak n$. Hence for arbitrary elements $Y, V, U'\in \nn$ one has
$$
\langle Y, [V, W]\rangle = \langle Y_{\mathfrak z}, [V_{\mathfrak v}, W_{\mathfrak v}]\rangle = \langle j(Y_{\mathfrak z})V_{\mathfrak v}, W_{\mathfrak v}\rangle,
$$
and the symplectic structure  on $TN$ can be written in the following way
$$
\Omega_{(p, Y)}((U, V), (U', V')) = \langle U, V'\rangle -\langle V, U'\rangle  + \langle j(Y_{\zz}) U_{\vv}, U_{\vv}'\rangle.
$$

In particular if $N$ is 2-step nilpotent and $f: TN \to \RR$ is a smooth map with gradient field $\grad_{(p,Y)} f=(U,V)$ then Lemma \ref{lema2} says
\begin{equation}\label{hamiltonian2}
X_f(p, Y) = (V,  j(Y_{\mathfrak z})V_{\mathfrak v}-U).
\end{equation}
Furthermore, for $f,g: TN \to \RR$  smooth maps with respective gradients given as $\grad_{(p,Y)} f=(U,V)$ and $\grad_{(p,Y)} g=(U',V')$, the Poisson bracket is
\begin{equation}\label{poisson2}
\{f,g\}(p,Y) = \la V',U\ra - \la V, U'\ra + \la j(Y_{\zz}) V'_{\vv}, V_{\vv}\ra.
\end{equation}

\section{The geodesic flow}

Let $N$ denote a Lie group with tangent bundle $TN$. 
In the following paragraphs we will denote by $\langle\cdot, \cdot\rangle$ both the left-invariant metric on $N$ and the product metric on $TN \simeq  N \times \mathfrak n$, that is, the product of the left-invariant metric on $N$ and the  inner product on $\nn$ induced by the left-invariant metric of $N$. 

The energy function $\En: TN \to \mathbb R$ is defined by
$$
\En(p, Y) := \frac{1}{2}\langle Y, Y\rangle_p=\frac12\la Y, Y\ra,
$$
since $\la \cdot, \cdot \ra$ is a left-invariant metric.

Notice that 
\begin{equation} \label{difen}
\begin{array}{rcl}
  d\En_{(p, Y)}(U, V) & = & \frac{d}{dt}\big|_0\En(p \exp(tU), tV + Y) \\
  & =  & \frac{d}{dt}\big|_0 \frac{1}{2}\left(t^2\langle V, V\rangle + 2t\langle Y, V\rangle + \langle Y, Y\rangle\right) \\
  & = & \langle Y, V\rangle.
\end{array}
\end{equation}

The gradient field  of the energy function $\grad\En$ is the vector field on $TN$  implicitly defined by the equation
$$
d\En_{(p, Y)}(U, V) = \langle\operatorname{grad}_{(p, Y)}\En, (U, V)\rangle,
$$
and thus the computations given in (\ref{difen}) imply
$$
\grad_{(p, Y)}\En = (0, Y). 
$$

The {\em geodesic field} on $TN$ is the Hamiltonian vector field of the energy function $X_{\En}$. The Equality (\ref{hamiltonianvf}) for  the energy function implies
$$X_{\En}(p, Y)=(Y, \ad^t(Y) Y),$$
and whenever $N$ is 2-step nilpotent we get
\begin{equation}\label{energy2s}
X_{\En}(p,Y)=(Y, j(Y_{\zz})Y_{\vv}).
\end{equation}

Let $(M, \la\,,\,\ra) $ denote a  complete Riemannian manifold with 
tangent bundle $TM$. For each $v\in TM$ and $t\in \RR$ define $\Gamma^t(v)=\gamma_v'(t)$
 the velocity at time $t$ of the unique geodesic with initial velocity $v$.
The fact that $M$ is complete implies that the geodesics of $M$ are defined on $\RR$. Thus for every  $v \in T_p M$ the curve on $TM$ given by $\Gamma^t(v)$ is defined  for all $t \in \RR$. 
 One can  check that $\Gamma^t\circ \Gamma^s=\Gamma^{t+s}$  
for all $s,t \in \RR$. The geodesic vector field is taken as the vector field on $TM$ with flow transformations  $\{\Gamma^t\}$.

On the other hand one has the flow of the Hamiltonian vector field $X_{\mathrm E}$ in $TM$ which is determined by the energy function $\mathrm E: TM \to \RR$, that is $E(v) =\frac12 < v, v >$ for all $v\in TM$. Now the flow of $X_E$ coincides with $\{\Gamma^t\}$. See for instance Section 5 in \cite{Eb2}.  

\medspace

This also applies on a Lie group $N$ endowed with a left-invariant Riemannian metric. Moreover in  \cite{Eb1,Eb2} one can see the definition and properties of the geodesic flow in the Lie algebra $\nn$ and the relationship between this with the geodesic flow defined above in terms of the Gauss map $G: TN \to \nn$. 

\medspace

Let $(M, \la \,,\,\ra)$ be a Riemannian manifold. The so called {\em first integrals} of the geodesic flow are the functions $f:TM \to \RR$ which Poisson commute with the energy function. This gives
$$0= \{f, \En\}(v)=df_v(X_{\En})= X_{\En}(f)= \frac{d}{ds} f \circ \Gamma^s(v),$$
from which it is clear that first integrals are functions which are constant along geodesics. 

First integrals for the geodesic flow 
can be constructed from Killing vector fields. In fact for a Killing vector field $X^*$ on $M$, the function $f_{X^*}: TM \to \RR$ given by
 \begin{equation}\label{fkil}
 f_{X^*}(v)=\la X^*, v\ra
  \end{equation}
\noindent becomes a first integral of the geodesic flow. In fact if $\gamma(t)$ is a geodesic on $M$ and $\nabla$ denotes the Levi Civita connection, then
   \begin{align*}
   \frac{d}{dt}\langle X^*(\gamma(t)), \gamma'(t)\rangle & = \langle\nabla_{\gamma'(t)} X^*, \gamma'(t)\rangle + \left\langle X^*(\gamma(t)), \frac{D}{dt}\gamma'(t)\right\rangle \\
   & = \langle\nabla_{\gamma'(t)} X^*, \gamma'(t)\rangle = 0
   \end{align*}
   since $(\nabla X^*)_{\gamma(t)}$ is skew-symmetric. So $\langle(X^*(\gamma(t)), \gamma'(t)\rangle$ is constant, and this implies that 
$f_{X^*}(v) = \langle X(\pi(v)), v\rangle$, is constant along the integral curves of the geodesic flow.
 
 \begin{lem}
   Let $M$ be a Riemannian manifold and let $X^*$ be a Killing field on $M$.  Then $f_{X^*}$ is a first integral of the geodesic flow.
 \end{lem}

In the setting of the previous section, let $N$ denote a Lie group and let  $f: TN \to \mathbb R$ be a smooth function with gradient field $\grad_{(p,Y)}(f) = (U, V)$. In view of the formulas in Lemma \ref{lema1} the function $f$ is a first integral of the geodesic flow if and only if
\begin{equation}\label{integ}
\begin{array}{rcl}
  0 & = & \Omega_{(p,Y)}(X_{\En}, X_f) = \Omega_{(p,Y)}((Y, \ad^t(Y) Y), (V, -U + \ad^t(V) Y))\\
  & = & \la Y, -U + \ad^t(V) Y \ra - \la \ad^t(Y) Y, V \ra + \langle Y, [Y, V]\rangle \\
  & = &  \la Y, -U + \ad^t(V) Y \ra - \la \ad^t(Y) Y, V \ra  + \la \ad^t(Y) Y, V \ra \\
	& = & \langle Y, -U + \ad^t(V) Y \rangle.
  \end{array}
\end{equation}
This proves the following result. 
\begin{lem}\label{lema2} Let $(N,\la\,,\,\ra)$ denote a  Lie group equipped with a left-invariant metric. Then the smooth function $f: TN \to \RR$ with gradient $$\grad_{(p, Y)} f= (U,V)$$ is a first integral of the geodesic flow if and only if for all $(p,Y)\in TN$ it holds
\begin{equation}\label{eq:integrability-condition}
  \langle Y, U\rangle = \langle Y, [V,Y]\rangle.
\end{equation}
In particular if $N$ is $2$-step nilpotent, then $f$ is a first integral if and only if 
$$\la Y, U\rangle = \langle j(Y_{\zz}) V_{\vv}, Y_{\vv}\ra.$$
\end{lem}

We shall say that $f:TN \to \RR$ is {\em invariant} if $f(p,Y)=f(g \cdot p, Y)$ for all $g, p\in N$, $Y\in \nn$, that is, $f$ is invariant under the left-action of $N$ into $TN$.

We make  use of the Equation (\ref{eq:integrability-condition}) to find invariant first integrals on  Lie groups.

\begin{proposition} \label{propcuad}
 Let  $(N, \la\,,\,\ra)$ be a Lie group endowed with a left-invariant metric and let $TN$ denote its tangent bundle with the product metric.
 \begin{itemize}
 \item Let $f_{Z_0}: TN \to \mathbb R$ be defined by 
 $$f_{Z_0}(p, Y) = \langle Y, Z_0\rangle.$$
  Then $f_{Z_0}$ is a first integral of the geodesic flow for all $Z_0 \in \mathfrak z$. Moreover, the family $\{f_{Z_i}:Z_i\in \zz\}$ is a commutative family of first integrals.
  \item Let $A: \mathfrak n \to \mathfrak n$ be a symmetric endomorphism of $\mathfrak n$ and let 
$$
g_A(p, Y) = \frac{1}{2}\langle Y, AY\rangle.
$$
Then $g_A$ is a  first integral of the geodesic flow if and only if
$$
0 = \langle Y, [AY, Y]\rangle.
$$
  \end{itemize}
\end{proposition}
\begin{proof}
  It is easy to see that $\grad_{(p, Y)}(f_{Z_0}) = (0, Z_0)$ and so the Hamiltonian vector field is $X_{f_{Z_0}} = (Z_0, 0)$. It follows from  Equality (\ref{eq:integrability-condition}) that $f_{Z_0}$ is a first integral of the geodesic flow. Moreover, 
$$
\Omega_{(p,Y)}(X_{f_{Z_0}}, X_{f_{Z_1}}) = \Omega_{(p,Y)}((Z_0, 0), (Z_1, 0)) = \langle Y, [Z_0, Z_1]\rangle = 0,
$$
which completes the proof of the first item.

We will deal now with quadratic first integrals of the geodesic flow. 

Let $A: \mathfrak n \to \mathfrak n$ denote a symmetric endomorphism of $\mathfrak n$ and let define
$$
g_A(p, Y) = \frac{1}{2}\langle Y, AY\rangle.
$$

An elementary calculation gives 
$$
dg_A|_{(p, Y)}(U, V) = \langle AY, V\rangle
$$ 
and hence 
$$
\grad_{(p, Y)}(g_A) = (0, AY).
$$ 
Then by Equation (\ref{eq:integrability-condition}), we have that $g_A$ is a first integral of the geodesic flow if and only if
$$
0 = \langle Y, [AY, Y]\rangle. \qedhere
$$
\end{proof}

\begin{remark} The last equality says that for $A$ symmetric on  a $2$-step nilpotent Lie algebra, the Poisson commutativity of $g_A$ with the energy function depends essentially on the restriction of $A$ to $\vv$.
\end{remark}

\begin{theorem} \label{teo1}  Let  $(N, \la\,,\,\ra)$ denote a $2$-step nilpotent Lie group endowed with a left-invariant metric and let $TN$ denote its tangent bundle with the product metric. Let $\nn$ be the Lie algebra of $N$. Let $A:\vv \to \vv$ denote a symmetric map with respect to $\la \,,\,\ra_{\vv}$ and let $g_A: TN \to \RR$ be given by
$$
g_A(p, Y) = \frac{1}{2}\langle Y, AY\rangle,
$$
where $A Z=0$ for all $Z\in \zz$. 
Then
\begin{enumerate}[(i)]
\item $g_A$ is a first integral of the geodesic flow if and only if for any basis $\{Z_1, \hdots , Z_m\}$ of $\zz$ one has
$$[J(Z_i), A]=0 \qquad \quad \mbox{ for all } i=1,\hdots , m.$$
\item Let  $g_A$ and $g_B$ be a pair of first integrals, then they Poisson commute  if and only if $$j(Z_i)AB = j(Z_i)BA \quad \mbox{ for all } i=1,\hdots m.$$
In particular if $\nn$ is non-singular, then 
$$\{g_A, g_B\} = 0\qquad \mbox{ if and only if } \qquad [A,B]=0.$$ 
\end{enumerate}
\end{theorem}
\begin{proof} (i) Proposition \ref{propcuad} says that $g_A$ is a first integral if and only if for $Y\in \nn$ one has 
$0=\la j(Y_{\zz})A Y_{\vv}, Y_{\vv}\ra$ which for a basis $\{Z_i\}$ of $\zz$ is equivalent to
\begin{equation}\label{aux}0=\la j(Z_i)A Y_{\vv}, Y_{\vv}\ra.
\end{equation}
In fact for the proof of $(\Rightarrow)$ take $Y= V + Z_i$ with $V\in \vv$ to get (\ref{aux}). The converse is clear since $j(Y_{\zz})=\sum_{i=1}^m a_i j(Z_i)$ for $a_i\in \RR$ and  $\{Z_1, \hdots, Z_m\}$ a basis of $\zz$. 

Now take $Y_{\zz}= V + V'\in \vv$ and write
$$0=\la j(Z_i) A(V+V'), V + V'\ra = \la j(Z_i) AV, V'\ra - \la A j(Z_i) V, V'\ra$$
which implies that $j(Z_i) A = A j(Z_i)$ for all $i=1, \hdots m$.

(ii) Assume that $g_A, g_B$ is a pair of first integrals of the geodesic flow (so that they satisfy (i) above). Then Lemma \ref{lema1} says that 
$$\begin{array}{rcl}
0 & = & \{g_A,g_B\} = \la Y, [AY,BY]\ra \\
& = & \la j(Y_{\zz}) A Y_v, B Y_{\vv}\ra \\
& = & \la B j(Y_{\zz}) A Y_{\vv}, Y_{\vv} \ra
\end{array}
$$
Firstly as in the proof of (i) the last equality is equivalent to $0= \la B j(Z_i) A Y_{\vv}, Y_{\vv}\ra$ for all $i=1,\hdots m$. Also as above $Bj(Z_i) A$ must be skew symmetric so that
$$(Bj(Z_i) A)^t= - A j(Z_i) B = - B j(Z_i) A \quad \mbox{ for all } i =1, \hdots , m.$$
Since the functions  $g_A, g_B$ are first integrals one has
$$j(Z_i) AB = j(Z_i) BA$$
which finally implies the statement. The last assertion follows from the last equality. 
\end{proof}

Let $N$ be a Lie group of dimension $n$. We shall say that the geodesic flow is completely {\em integrable} if there exists $n$ smooth functions $f_1, f_2, \hdots, f_n:TN \to \RR$ such that
\begin{itemize}
\item the gradients of $f_1, f_2, \hdots, f_n$ are linearly independent on an open dense subset of $TN$;
\item $0=\{f_i, f_k\}=\{f_i, E\}$ for all $i,k=1, \hdots, n$.
\end{itemize}

We now consider a sufficient non-integrability criterion by Butler \cite{Bu2}:

\begin{defn} Let $\nn$ be a 2-step nilpotent Lie algebra.
\begin{itemize}
\item[(i)]
 For $\lambda \in \nn^*$, let  $\nn_{\lambda}:= \{X\in \nn / \ad^*(X) \lambda=0\}$.
\item[(ii)]  A $\lambda \in \nn^*$  is called regular if $\nn_{\lambda}$  has minimal dimension.
\item[(iii)] $\nn$ is called non-integrable if there exists a dense open subset $\mathcal W$ of $\nn^*\times \nn^* $ such that for
each $(\lambda, \mu) \in \mathcal W$, both $\lambda$ and $\mu$  are regular and $[\nn_{\lambda}, \nn_{\mu}]$ has positive dimension.
\end{itemize}
\end{defn}

In view of the metric on $\nn$ any element $\lambda\in \nn^*$ can be realized in the form $\lambda=\la V+Z, \cdot\ra$ for unique $V\in \vv$ and $Z\in \zz$. By identifying  $\lambda\in\nn^* \longleftrightarrow V+Z\in \nn$  we
 get
$$
\begin{array}{rcl}
\nn_{V+Z} & = & \{X\in \nn \, : \, \la V + Z, \ad(X) \cdot \ra=0\} \\
 & = & \{X\in \nn \, : \, \la V + Z, \ad(X)(V' + Z') \ra=0\, \mbox{ for all } V'+Z'\in\nn\}\\
& = & \{X\in \nn \, : \, \la Z, [X, V'] \ra=0\, \mbox{ for all } V'\in\vv\},
\end{array}
$$
from which it is clear that $\zz \subseteq \nn_{\lambda}$ for all $\lambda \in \nn^*$. Moreover if we denote $X=\tilde{V}+\tilde{Z}$ then
$$
\tilde{V}+\tilde{Z}\in\nn_{V+Z}\Leftrightarrow \la j(Z) \tilde{V}, \cdot \ra=0  \Leftrightarrow \la \tilde{V}, j(Z) V' \ra=0\, \mbox{ for all } V' \in \vv.
$$

Hence if $\nn$ is non-singular then $\tilde{V}=0$ since 
$\Im j(Z)=\vv$. In this situation $\nn_{V+Z}$ has minimal dimension for  $Z\neq 0$ and $[\nn_{\lambda}, \nn_{\mu}]=0$. One gets the following result. 

\begin{lem} Let $\nn$ be a non-singular 2-step nilpotent Lie algebra, then $\nn$ cannot be a non-integrable Lie algebra.
\end{lem}

Thus non-singular Lie algebras are candidate to have a completely integrable geodesic flow in view of the next result.

\smallskip

\begin{utheorem}[\cite{Bu2}, Theorem 1.3] Let $\nn$ be a non-integrable $2$-step nilpotent Lie algebra with  associated simply connected Lie group $N$. Assume that there exists a discrete cocompact subgroup $\Lambda$ of $N$. Then for any such $\Lambda$  and any left-invariant metric $g$ on $N$, the geodesic flow of $(\Lambda \backslash N, g)$ is not completely integrable.
\end{utheorem}

An example of a non-singular Lie algebra is the Heisenberg Lie algebra $\hh_n$ corresponding to the Heisenberg Lie group $\Heis_n$. We shall see in the next section that the geodesic flow on any quotient $\Lambda \backslash \Heis_n$ is completeley integrable for any induced left-invariant metric. 

\begin{remark} Topological obstructions to integrability on two-dimensional surfaces were given in \cite{Kol, Koz}. In any dimension on non-simply-connected manifolds by Taimanov \cite{Ta1}.
\end{remark}

\section{On the first integrals of the geodesic flow on the Heisenberg manifolds}

A Riemannian {\em Heisenberg manifold} is a compact manifold given as a quotient $\Lambda \backslash \Heis_n$ where $\Heis_n$ denotes the Heisenberg Lie group of dimension $2n+1$ and $\Lambda < \Heis_n$ is a discrete cocompact subgroup. This notation coincides with that one in \cite{GW}. 

The  Heisenberg Lie group $\Heis_n$  is the simply connected Lie group constructed over the smooth space  $\RR^{2n+1}$ together with the multiplication map: 
$$
(v, z)(v', z') = \left(v+v', z+z'- \frac{1}{2}v^{\tau} J v'\right) \qquad v, v' \in \RR^{2n},\,  z\in \RR, 
$$
where $J$ is the real linear map representing the multiplication by $i$ once we identify $\RR^{2n}$ with 
$\CC^n$. It is clear that  a basis of left-invariant vector fields is given for $k=1, \hdots, n$ by
$$Z_1(p)= \partial_z, \qquad X_{2k-1}(p)=\partial_{x_{2k-1}}-\frac12 p_{2k}\partial_z,\qquad X_{2k}(p)=\partial_{x_{2k}}+\frac12 p_{2k-1}\partial_z,$$
where $x_i$ are the canonical coordinates on $\RR^{2n}$ and with $\partial_y$ we denote the partial derivative 
$\frac{\partial}{\partial y}$. So any left-invariant vector field $U$  on $\Heis_n$ has the form
$$U(p)= \sum_{i=1}^{2n} a_i X_i + c Z_1= \sum_{i=1}^{2n} a_{i} \partial_{x_i} + 
\left(c+\frac12\sum_{i=1}^ n(-a_{2i-1}p_{2i}+a_{2i}p_{2i-1})\right)\partial_z.$$

Note that $\sum_{i=1}^ n(-a_{2i-1}p_{2i}+a_{2i}p_{2i-1})=\omega(P, A)$ where $\omega$ is the canonical symplectic form of $\RR^{2n}$ and $P=(p_1, p_2, \hdots, p_{2n}),\, A=(a_1, a_2, \hdots , a_{2n})$. 

Let $\la \,,\,\ra$ be the inner product on $\hh_n$ which turns the set 
$\{X_1, X_2, \ldots, X_{2n}, Z_1\}$ into  an orthonormal basis. In canonical coordinates of $\RR^{2n+1}$ we get the following metric:
$$dz^2 + \sum_{k=1}^{2n} [(-1)^{k+1} y_k\, dz dx_k+  (1+\frac14 y_k^2) dx_k^2]+\frac12 \sum_{k\neq s, 1}^{2n} (-1)^{k+s} y_k y_s dx_k dx_s$$

where $y_t = \left\{ \begin{array}{ll}
x_{2l} & \mbox{ for } t=2l-1\\
x_{2l-1} & \mbox{ for } t=2l.
\end{array}
\right.$

We also denote by 
$\langle\cdot, \cdot\rangle$ the left-invariant metric on $\Heis_n$. Then the restriction of $\la \cdot, \cdot \ra$ to the Lie algebra $\hh_n$ gives a decomposition
$$
\hh_n = \mathfrak v \oplus \mathfrak z \qquad \text{ as orthogonal sum of vector subspaces},
$$
where $\mathfrak v$ is the subspace spanned by $X_1, X_2, \ldots, X_{2n}$ and $\mathfrak z = \mathbb RZ_1$ is the 
center of $\hh_n$. 
For a generic left-invariant vector field $U$ denote by $U_{\vv}$ and $U_{\zz}$ the respective projections onto 
$\vv$ and $\zz$. One can see that the integral curve of $U$ at the point $p=(p_v, p_z)$ is 
given by $$\gamma_p(s)=(p_v+sU_{\vv}, p_z+ sU_{\zz}-\frac12 s p_v^{\tau}J U_{\vv})$$ via natural identifications 
at the corresponding vector spaces.

It is the clear that the exponential map $\exp:\hh_n \to \Heis_n$ is given by
$$\exp(U) = U$$ via identifications: at the left-side $U$ is given in terms of a left-invariant basis, while at 
the right side the vector $U\in \RR^{2n+1}$ is given in canonical coordinates.

Under the Lie bracket on $\mathfrak X(\Heis_n)$ the set of left-invariant vector fields give rise to the Lie algebra  $\hh_n$ called  
the Heisenberg Lie algebra whose  basis 
$\{X_1, X_2, \ldots, X_{2n}, Z_1\}$ satisfies the non-zero Lie bracket relations
$$
[X_{2i - 1}, X_{2i}]= Z_1, \qquad i = 1, \ldots, n
$$
 so that 
 for $Z \in \mathfrak z$, the skew-symmetric endomorphism $j(Z) \in \sso(\mathfrak v)$ as in (\ref{br}) takes the form
$$
j(Z) = \langle Z, Z_1\rangle J $$ where $J:=j(Z_1)$ has the following matrix 
$$
\begin{pmatrix}
  0 & -1 & & & \\
  1 & 0 & & & \\
  & & \ddots & & \\
  & & & 0 & -1 \\
  & & & 1 & 0
\end{pmatrix}
$$
in the basis $\{X_1, X_2, \ldots, X_{2n}\}$ of $\mathfrak v$. 

Since the dimension of the center of $\hh_n$ satisfies $\dim  \zz(\hh_n)=1$, we can only define one independent linear first integral for the geodesic flow:
$$
f_{Z_1}(p, Y) = \langle Y, Z_1\rangle.
$$
As we said before, we have that
$$
\grad_{(p, Y)}f_{Z_1} = (0, Z_1), \qquad X_{f_{Z_1}}(p, Y) = (Z_1, 0).
$$

As known, the isometry group of a 
simply connected nilmanifold $M$ is the semidirect product $Iso(M)=K\ltimes N$ where 
$N$ is the nilradical of the isometry group which acts simply and transitive on $M$. Naturally $M$ can be 
identified with $N$ so that the homogeneous metric on $M$ turns into a left-invariant metric on $N$. Moreover the isotropy 
subgroup $K$ consists of isometric automorphisms. See \cite{Wo} for more details. 

Let $K$ be the subgroup of isometric automorphisms of the Heisenberg Lie group $\Heis_n$, its  Lie algebra $\mathfrak k$ is given by
\begin{equation}
  \label{isomh}
  \mathfrak k =\{ B\in \mathfrak{so}(\vv, \la\,,\,\ra_{\vv}) \quad \mbox{ such that } \,\, [J, B]=0\},
 \end{equation}
see for instance \cite{La}. 

It is clear that the action of $\Heis_n$ on itself by isometries is given by the translations on the left.

\medspace

Let us search for the quadratic homogeneous polynomials which are first integrals of the geodesic flow.  For a symmetric map $A:\vv \to \vv$ the quadratic polynomial given by
 $$g_A(p, Y)=\la AY, Y\ra$$
 is a first integral of the geodesic flow on the Heisenberg Lie group $\Heis_n$ if 
 $$[J, A]=0 \qquad \qquad \mbox{ by Theorem \ref{teo1}}.$$
 
 Take $A:\vv \to \vv$ a symmetric map such that $[J,A]=0$ then the map given as $B:=JA\in \mathfrak{so}(\vv \la\,,\,\ra_{\vv})$ 
 also satisfies
 $$[J,B]=0$$
 so that $B$ belongs to the isometry algebra of $\Heis_n$. 
 In fact, as explained above $Iso(\Heis_n)=K\ltimes \Heis_n$ where $K$ denotes the isotropy subgroup consisting of isometric 
 automorphisms. Its Lie algebra is $\mathfrak{iso}(\Heis_n)=\mathfrak k \oplus \hh_n$ where  $\mathfrak k$ is the Lie algebra of $K$ and it 
 consists of skew-symmetric derivations, see (\ref{isomh}).

 Conversely take $B\in \mathfrak k$ and define a symmetric map $A:\vv \to \vv$ by $A:=JB$. Then $g_A$ is 
 a first integral  of the geodesic flow. 
 
 \begin{theorem} There is a bijection $\psi$ between the set of quadratic first integrals of the geodesic flow and 
 the Lie subalgebra of skew-symmetric derivations of $\hh_n$ given by
  \begin{equation}\label{simde} A \quad \longrightarrow \quad \psi(A) : = JA
	\end{equation}
  so that $\{g_{A_1}, g_{A_2}\}=0$ if and only if $[\psi(A_1), \psi(A_2)]=0$ in $\mathfrak{so}(\vv, \la\,,\,\ra_{\vv})$.
 \end{theorem}

 \begin{remark}
  Note that $\psi$ is not a Lie algebra morphism. 
 \end{remark}

\begin{remark} The Gauss map $G: T \Heis_n \to \hh_n$ is an anti Poisson map. The metric $\la\,,\,\ra$  and the Lie bracket on $\hh_n$ define a Poisson structure on the Heisenberg Lie algebra $\hh_n$. For the energy function on $\hh_n$ given by $\bar{E}(v)=\frac12 \la v, v \ra$ one also has its Hamiltonian vector field called the geodesic vector field, which corresponds in a natural way to the geodesic flow means of the Gauss map $G: T \Heis_n \to \hh_n$. 
In fact if $f$ is an invariant function on $T \Heis_n$ which corresponds to a function $\bar{f}:\hh_n \to \RR$ then $G(X_f(p,v))=-X_{\bar{f}}(v)$ where $X_f$ denotes the Hamiltonian vector field for $f$ while $X_{\bar{f}}$ for $\bar{f}$. If $\Gamma^t$ denotes the geodesic flow for  $X_E$ and $\bar{\Gamma}^t$ the flow of $X_{\bar{E}}$ then $G\circ \Gamma^{-t}=\bar{\Gamma}^t \circ G$ for all $t\in \RR$.
See \cite{Eb2} for more details. A formula for geodesics on the Heisenberg Lie group can be found in \cite{Eb}. 

The coadjoint orbits on the Heisenberg Lie algebra $\hh_n$ were described in \cite{Ov} were making use of the Adler-Kostant-Symes scheme they were used for a study of the motion of $n$ harmonic oscillators near an equilibrium position.

\end{remark}

From the result above it is clear that to find the symmetric maps $S:\vv \to \vv$ such that $[J,S]=0$ is equivalent to determine the skew-symmetric maps $T\in \sso(\vv)$ such that $[J,T]=0$. Choose $X_1, X_3, \hdots, X_{2n-1}, X_2, X_4, \hdots, X_{2n}$ the ordered basis of $\vv$ and write $T$ and $J$ in this basis. A general matrix $T\in \sso(2n)$ of the form
$$\left( \begin{matrix}
Q_1 & Q_2 \\
-Q_2^{\tau}  & Q_3 \end{matrix} \right) $$
with  $Q_1, Q_3\in \sso(n)$ real submatrices of $T$, commute with $J$
$$J=\left( \begin{matrix}
0 & -\rm{\bf 1} \\
\rm{\bf 1} & 0 \end{matrix}
\right),$$
 $[J, T]=0$ if and only if $Q_1\in \sso(n)$, $Q_1=Q_3$ and $Q_2=Q_2^{\tau}$.

\

Take the quadratic first integrals given by the symmetric linear maps $A_i:\vv \to \vv$ which correspond to
\begin{equation}\label{syme}
A_iY = \langle Y, X_{2i - 1}\rangle X_{2i - 1} + \langle Y, X_{2i}\rangle X_{2i}, \qquad i \ge 1.
\end{equation}

It is not hard to prove that $[J, A_i]=0$ and $[A_i, A_j]=0$ for all $i, j=1, \hdots, n$, which gives $\{g_{A_i},g_{A_j}\}(p,Y)=0$ by Theorem \ref{teo1}.

Note that the first integrals $f_{Z_1}$, $g_{A_i}$ ($i \ge  0$), are algebraic, that is invariant under 
the action of $\Heis_n$ on $T \Heis_n \simeq \Heis_n \times \hh_n$.

Recall that Killing 
vector fields built the Lie algebra of the isometry group of a given homogeneous nilmanifold. 
So for the Heisenberg Lie group $\Heis_n$ we have
$$Iso(\Heis_n)= K \ltimes \Heis_n$$
where $\Heis_n$ is the normal subgroup of $Iso(\Heis_n)$ corresponding to the  translations on the left by elements of $\Heis_n$, $K$ is the subgroup of orthogonal automorphisms and the action of $K$ on $\Heis_n$ is given by the evaluation map.

 Let $T\in \sso(\vv)$ such that $[J,T]=0$ and $\rho_T$ be the automorphism of $\Heis_n$ given by 
 $$ \rho_T (p)= (e^T p_{v}, z)$$
 where $p_v$ is the projection of $p$ onto $\RR^{2n}$ and $e^B$ is the usual exponential map: $e^B=\sum_{i=0} \frac1{k!}B^k$. Since $\rho_T$ is a group homomorphism the  fact that it is  an isometry at the  identity element implies that  $\rho_T$ is an isometry of $\Heis_n$. 
 
 For this isometry the corresponding Killing vector field computed as 
$$X_T^*(p)=\frac{d}{ds}\bigg|_0 \rho_T (p)= \frac{d}{ds}\bigg|_0(e^{sT} p_{v}, z)$$
  is given by
 \begin{equation}\label{X_T}
X_T^*(p)=T W_{\vv} - \frac12 \la T W_{\vv}, J W_{\vv}\ra Z_1
\end{equation}
 where $T\in \sso(\vv)$,  so that $\exp W=p$  for  the exponential map $\exp: \hh_n \to \Heis_n$ and $W_{\vv}$ denotes 
 the projection of $W$ onto $\vv$ with respect to the orthogonal decomposition $\hh_n=\vv\oplus \zz$.
 
 The Killing vector fields which generate $\hh_n$ as subalgebra of the isometry algebra corresponds to the right-invariant vector fields given by
 $$Z_1^*(p)= \partial_z, \qquad X_{2i-1}^*(p)=\partial_{x_{2i-1}}+\frac12 p_{2i}\partial_z,\qquad
 X_{2i}^*(p)=\partial_{x_{2i}}+\frac12 p_{2i-1}\partial_z.$$

Explicitely, for $X_k^*$ with,  $k=0,\hdots, 2n$,  corresponding to the left-translations isometries, we get the next  first integrals - see (\ref{fkil}):
 $$F_{2i-1}(p,Y)= \la X_{2i-1} + \la W_{\vv}, X_{2i}\ra Z_1, Y\ra \qquad \mbox{ for } X_{2i-1}^*, \, i=1, \hdots, n$$
 $$F_{2i}(p,Y)= \la X_{2i} - \la W_{\vv}, X_{2i-1}\ra Z_1, Y\ra \qquad \mbox{ for } X_{2i}^*, \, i=1, \hdots, n$$
 which for $k = 1, \ldots, 2n$ can be rewritten as
\begin{equation}\label{Fk}
 \begin{array}{rcl}
 F_k(p, Y) & = &\langle Y - J(Y_{\mathfrak z})W_{\mathfrak v}, X_k\rangle \\
  & = & \langle Y, X_k\rangle - \langle J(Y_{\mathfrak z})W_{\mathfrak v}, X_k\rangle,
\end{array}
\end{equation}
 while for $Z_1^*$ one has the function
$$f_{Z_1}(p,Y)=\la Z_1, Y\ra$$
which was introduced in the previous section, see Proposition \ref{propcuad}. 

For the Killing vector field $X_T^*$ corresponding to an element $T$ of the isotropy subgroup $K\subset Iso(\Heis_n)$ one has the first integral 
\begin{equation}\label{FT}
F_T(p, Y)= \la T W_{\vv}, Y_{\vv}\ra - \frac12 \la AW_{\vv}, W_{\vv}\ra \la Z_1, Y\ra
\end{equation}
where $A$ is the symmetric map of the bijection (\ref{simde}) for  $T=JA$.

For these functions one gets the following gradient vector fields
\begin{equation}\label{gradientsFi}
 \begin{array}{rcl}
\grad_{(p,Y)} F_T & = & (-\la Z_1,Y\ra A W_{\vv}- TY_{\vv}, TW_{\vv} -\frac12 \la A W_{\vv},W_{\vv} \ra Z_1)\\

  \grad_{(p, Y)}F_k & = & (\langle Y, Z_1\rangle  J X_{k}, X_{k} + \langle W, J X_k\rangle Z_1) \mbox{ for } k=1, \hdots, 2n
 \\
 \grad _{(p, Y)} f_{Z_1} & = & (0, Z_1)
 \end{array}
\end{equation}
which shows that the gradients of $f_{Z_1}, F_k$ for $k=1, \hdots 2n$ are linearly independent.

The formula in (\ref{poisson2}) says that for $j, k \in \{0, \hdots, 2n\}$ the Poisson bracket is given by
\begin{equation}
 \{F_j, F_k\}(p,Y)= \la Y, Z_1\ra \la X_k, JX_j\ra  
\end{equation}
so that the only non-trivial brackets are
\begin{equation}\label{poisson-fint1}
\{F_{2i-1}, F_{2i}\}= f_{Z_1}.
 \end{equation}

Let $\rho_{T_a}, \rho_{T_b}$ denote a pair of isometries as above with $T_a=JA$, $T_b=JB$. Recall that $T_a, T_b\in \sso(v)$ and $0=[J, T_i]=[J, T_j]$.  Take the corresponding functions $F_{T_i}, F_{T_j}$ as in (\ref{FT}) with corresponding gradient vector fields 
$\grad_{(p,Y)}F_{T_i}=(U,V)$ and $\grad_{(p,Y)}F_{T_j}=(U',V')$. Then by Lemma \ref{lema1} one has 

\vskip .1pt

$\{F_{T_a}, F_{T_b}\}(p,Y)=0$ if and only if $\la V', U\ra- \la V, U'\ra + \la Y, [V', V]=0\ra $. 

\vskip .1pt

In particular for  $(p, Z_1)$ one gets

$$\begin{array}{rcl} 
 \{F_{T_a}, F_{T_b}\}(p,Y) & = &  -\la T_b W_{\vv}, A W_{\vv}\ra + \la T_a W_{\vv}, B W_{\vv}\ra + \la J T_b W_{\vv}, T_a W_{\vv}\ra \\
& = &  -\la JAB W_{\vv},  W_{\vv}\ra + \la JBAW_{\vv},  W_{\vv}\ra +  \la J^2 B W_{\vv}, J A W_{\vv}\ra \\
&=&  -\la  JAB W_{\vv},  W_{\vv}\ra + \la JBA W_{\vv},  W_{\vv}\ra +  \la JAB W_{\vv},  W_{\vv}\ra, \\
\end{array}
$$ 
so that as in the proof of Theorem \ref{teo1}  $ \{F_{T_a}, F_{T_b}\}(p,Y)=0$ implies $[A,B]=0$ for $T_a=JA, T_b=J B$. 

Conversely it is not hard to see that $\{F_{T_a}, F_{T_b}\}(p,Y)=0$ whenever $[A,B]=0$ for all $(p,Y)$.

\begin{lem}  \label{lema4} Let $T_a=JA$ and $T_b=JB$ denote skew symmetric derivations. Let $F_{T_a}, F_{T_b}: T \Heis_n \to \RR$ be the functions  defined above. It holds
$$\{F_{T_a}, F_{T_b}\}(p,Y)=0 \, \mbox{ if and only if } [A,B]=0. $$
Moreover, if $[A,B]=0$ then $\{F_{T_a}, g_{B}\}(p,Y)=0$ whenever  $A,B$ are symmetric maps.
\end{lem}

Making use of (\ref{gradientsFi}) the other Poisson brackets at every point $(p,Y)$ for  $p=\exp W$, satisfy
\begin{itemize}
\item $\{F_{T_a}, g_B\}(p,Y)=\la J A B Y_{\vv}, Y_{\vv}\ra$ for $T_a=JA$.

\item Let $T_a=JA$, $T_b=JB$ for $A,B$ symmetric. 
$$\begin{array}{rcl}
\{ F_{T_a}, F_{T_b}\} (p,Y) & = & \la [T_a, T_b] W_{\vv}, Y_{\vv}\ra + \la Y, Z_1\ra \la J BA W_{\vv}, W_{\vv}\ra\\
& = &\la [T_a, T_b] W_{\vv}, Y_{\vv}\ra + \la Y, Z_1\ra \frac12 \la (J BA - JAB)W_{\vv}, W_{\vv}\ra\\
&=& F_{[T_a,T_b]}.
\end{array}
$$ 
Notice that $\frac12 (JAB - JBA)$ is the symmetric part of $JAB$.

\item $\{F_T, F_k\}(p,Y)= F_k(p, T Y_{\vv}+ Y_{\zz})$, for $k=1, \hdots, 2n$.
\item   $\{f_{Z_1}, F_i\}=0$ for all $i=1, \hdots, n$.

\end{itemize}

Let $A_i: \vv \to \vv$ the symmetric maps as in (\ref{syme}) and let $T_i=JA_i$ the corresponding skew-symmetric maps. Then from the equations of the gradients above one can see that $F_{T_i}, g_{A_j}$ is a set of $2n$ linearly independent functions on an open dense subset of $T \Heis_n$.

Note that the set $\{JA_i\}$ correspond to a Cartan subalgebra of $\sso(\RR^{2n})$ which in particular gives an abelian subalgebra of skew-symmetric derivations. 

 We are now in conditions to prove the following result.

\begin{theorem}\label{heisenberg-integrable}
   The geodesic flow on $T \Heis_n$ is completely integrable in the sense of Liouville. In fact the set
	$$\mathcal G=\{ E\} \cup \{g_{A_i}\}_{i = 1}^n \cup \{F_{T_k}\}_{k = 1}^n$$ gives a family of commuting first integrals of the geodesic flow. 
	
	The  sets
  \begin{enumerate}
  \item 
    $$
    \mathcal F = \{f_{Z_1}\} \cup \{g_{A_i}\}_{i = 1}^n \cup \{F_{2k - 1}\}_{k = 1}^n
    $$
  \item  and
    $$
    \mathcal F' = \{f_{Z_1}\} \cup \{g_{A_i}\}_{i = 1}^n \cup \{F_{2k}\}_{k = 1}^n
    $$
     \end{enumerate}
    give two independent commuting families of first integrals of the geodesic flow.
 
\end{theorem}

\begin{proof} Lemma \ref{lema4} says that

\begin{itemize}
\item  $\{F_{T_i}, F_{T_j}\}=0$ and $\{F_{T_i}, g_{A_j}\}=0$ for $T_i=JA_i$ with $i=1, \hdots, n$ and 

\item $\{g_{A_i}, g_{A_j}\}=0$ since $[A_i, A_j]=0$,
 \end{itemize}
and since all of them are first integrals the fact that this is a completely integrable set follows from the condition of linearly independence of the set on a open dense subset of $T \Heis_n$.

For the other families by making use of the information already given we have that 
\begin{itemize}
 \item $\{F_{2k - 1}, F_{2j-1}\}=0=\{F_{2k}, F_{2j}\}$ for all $k, j=1, \hdots, n$ and $k\neq j$.
\end{itemize}
We should compute the other Poisson brackets.
  $$
	\begin{array}{rcl}
    \{f_{Z_1}, g_{A_i}\}(p, Y) & =  &\langle Z_1, j(Y_{\mathfrak z})A_iY\rangle + \langle Y,  [Z_1, A_iY]\rangle = 0, \\
    \{g_{A_i}, F_{2j - 1}\}(p, Y)  
    & = &-\langle j(Y_{\mathfrak z})A_iY, X_{2j - 1}\rangle + \langle Y, [A_iY, X_{2j - 1}]\rangle = 0, \\
    \{g_{A_i}, F_{2k}\}(p, Y)  
    &  = & -\langle j(Y_{\mathfrak z})A_iY, X_{2k}\rangle + \langle Y, [A_iY, X_{2k}]\rangle = 0,
    \end{array}
		$$
   which completes the proof of the theorem.
\end{proof}

\begin{remark} It turns out that the families given in \cite{Bu1} are of the same type as the families in $\mathcal F$  and $\mathcal F'$ above. While the first  integrals $F_{T_j}$ are new. 

Notice that in the family $\mathcal G$  all the first integrals are obtained from the Lie algebra of the isotropy subgroup  $ K \subset Iso(\Heis_n)$. 
\end{remark}

\begin{remark}
  In the simply connected case we are working on all the first integrals introduced in Theorem \ref{heisenberg-integrable} are analytic. As we shall see later, this cannot be achieved for compact quotients of $\Heis_n$.
\end{remark}

The Poisson brackets of the functions $\{F_T, F_k, f_{Z_1}\}$ for $T\in \kk$, $F_k$, $k=1, \hdots, 2n$ give the proof of the next result.

\begin{theorem} \label{thm33}
 The linear morphism between the Killing vector fields on $\Heis_n$ and its image in $C^{\infty}(T\Heis_n)$ given by
 $$X^* \quad \longrightarrow \quad f_{X^*}$$
 as in (\ref{fkil}) builds a Lie algebra isomorphism between the isometry Lie algebra of $\Heis_n$ and its
 image equipped with the Poisson bracket. 
 \end{theorem}

\begin{proof} In fact, a Killing vector field $X^*$ can be written as $X^*=X_T^*+X_U^*$ for $T+U\in \kk \oplus \hh_n$. Moreover
$U(p)= V(p) + z Z_1^*$ and $V(p)=\sum_k s_k X_k^*(p)$. Thus
$$f_{X^*}(p,Y)=F_T(p,Y)+ \sum_k s_k F_k(p,Y) + z f_{Z_1}(p,Y)
$$
so that $X^* \to f_{X^*}$ is a Lie algebra homomorphism. It is injective: in fact assume
\begin{equation}\label{inj}
0=F_T(p,Y)+ \sum_k s_k F_k(p,Y) + z f_{Z_1}(p,Y)\quad \mbox{  for all $(p,Y)$.}
\end{equation}
By taking $p=\exp Z_1$, $Y=Z_1$ one gets $z=0$.

\vskip .1pt

Take $Y=Z_1$ so that $F_T(p,Y)=-\frac12 \la AW_{\vv}, W_{\vv}\ra$ and $F_k(p,Z_1)=- \la JW_{\vv}, X_k \ra$. Thus equality (\ref{inj}) becomes
$$0=-\frac12 \la AW_{\vv}, W_{\vv}\ra - \sum_k s_k \la JW_{\vv}, X_k \ra.$$

If $\la A X_j, X_j\ra=0$ for all $j=1, \hdots, 2n$ then $A=0$ and so $F_T=0$. Choosing $W= X_l$ one gets $ 0 = \sum_k s_k \la JX_l, X_k \ra = \pm s_l$ which says $s_l=0$ for every $l$.

If $\la A X_j, X_j\ra\neq 0$ then there is $j$ such that $a_{jj}=\la AX_j, X_j\ra \neq 0$. Take $W= t X_j$ for $t\in \RR$. Then
(\ref{inj}) becomes
$0=-\frac12 a_{jj}t^2 - \pm s_jt= t( -\frac12 ta_{jj}-\pm s_j)$ which should holds for every $t$. The sign $\pm$ depends on the parity of $j$. Thus $a_{jj}=0$ and $s_j =0$ for all $j$. And $X^* \to f_{X^*}$ is injective onto its image.
\end{proof}
 
 So far we know this is the first example among nilpotent Lie groups of this isomorphism. The question appeared in \cite{Th}. 
 It could be interesting to know if there are more examples of this situation among nilpotent Lie groups and if is there a 
 relationship with the complete integrability of the geodesic flow.

\begin{remark} The metric considered in this section make of $N$ a naturally reductive Riemannian space. That means that there exists a Lie group of isometries $G$ acting transitively on $N$ such that its Lie algebra $\ggo$ splits as $\ggo=\hh\oplus \mm$ where $\hh$ represents the Lie algebra of the isotropy subgroup and $\mm$ is a $\Ad(H)$-invariant complement such that
$$\la x, [y,z]_{\mm} \ra + \la y, [x,z]_{\mm}\ra=0 \qquad \mbox{ for all }x,y,z\in \mm.$$
The Heisenberg Lie group does not correspond to the examples  in \cite{Th}.
\end{remark}

\subsection{ Riemannian Heisenberg manifolds} 

Let $\Heis_n$ the $(2n + 1)$-di\-men\-sio\-nal Heisenberg group endowed with the standard Riemannian metric, which is defined in the beginning of this section. For each $n$-tuple $r = (r_1, \hdots, r_n) \in (\mathbb Z^+)^n$ such that $r_1 \mid r_2 \mid \cdots \mid r_n$, we define
\begin{equation}\label{eq:lattice}
  \Lambda_r = \{(v,z):  v = (x, y) \text{ with } x\in r\mathbb Z^n,\, y\in 2\mathbb Z^n,\, z\in\mathbb Z\},
\end{equation}
where $x = (x_1, \ldots, x_n) \in r\mathbb Z^n$ means that $x_i \in r_i\mathbb Z$ for all $i = 1, \ldots, n$.

It follows from \cite{GW} that this family classifies the cocompact discrete subgroups of $\Heis_n$ up to isomorphism. This is explained in the following remark.

\begin{remark} 
  In several works the Heisenberg Lie group is defined started with the triangular $(2n + 2) \times (2n + 2)$ real matrices of the form
  $$\gamma(x,y,t) = \left( \begin{matrix}
    1 & x & t\\
    0 & I_n & y^{\tau} \\
    0 & 0 & 1
  \end{matrix}
  \right),
  $$
  where $x, y \in \RR^n$ and  $I_n$ is the $n \times n$ identity matrix.
  The map $\Phi$ defined as $\Phi(\gamma(x,y,t)) = (x, y, t - \frac{1}{2} x \cdot y)$ gives an isomorphism with the Heisenberg Lie group defined above in the first paragraphs of this section. In \cite{GW} it is proved that any lattice on the Heisenberg group, with this matrix presentation, is isomorphic to one of the form
  $$
  \Gamma_r = \{\gamma(x, y, t): x \in r\mathbb Z^n, \, y \in \mathbb Z^n, \, t \in \mathbb Z\},
  $$
  where $r = (r_1, \ldots, r_n) \in (\mathbb Z^+)^n$ is such that $r_1 \mid \cdots \mid r_n$. Moreover, $\Gamma_r$ is isomorphic to $\Gamma_s$ if and only if $r = s$. Recall that, $\Phi^{-1}(\Lambda_r) \subset \Gamma_r$. This contention is strict, but one can still prove that $\Lambda_r$ is isomorphic to $\Gamma_r$. For the sake of completeness we include the proof of this fact in the following remark. 
\end{remark}

\begin{remark}
  Recall that the center $Z(\Gamma_r)$ of $\Gamma_r$ is $\{\gamma(0, 0, t): t \in \mathbb Z\}$ for all $r$. So $Z(\Gamma_r)$ is cyclic with two generators $\gamma(0, 0, 1)$ and $\gamma(0, 0, -1)$. As it follows from \cite{GW}, for each $i = 1, \ldots, n$ we have that 
  $$
  \gamma(r_ie_i, 0, 0)\gamma(0, e_i, 0)\gamma(r_ie_i, 0, 0)^{-1}\gamma(0, e_i, 0)^{-1} = \gamma(0, 0, 1)^{r_i},
  $$
  and for all $s \in (\mathbb Z^+)^n$ such that $s_1 \mid \cdots \mid s_n$, $s \neq r$, there are no elements $\gamma_1, \gamma_2, \ldots, \gamma_{2n} \in \Gamma_s$ such that 
  $$
  \gamma_{2i - 1}\gamma_{2i}\gamma_{2i - 1}^{-1}\gamma_{2i}^{-1} = \gamma(0, 0, \pm 1)^{r_i}.
  $$
  This shows, in particular, that $\Gamma_r$ is not isomorphic to $\Gamma_s$ if $r \neq s$. Now, from the classification theorem in \cite{GW}, we have that $\Lambda_r$ is isomorphic to $\Gamma_s$ for some $s$. But 
  $$
  (r_ie_i, 0, 0)(0, 2e_i, 0)(-r_ie_i, 0, 0)(0, -2e_i, 0) = (0, 0, r_i) = (0, 0, 1)^{r_i},
  $$
  where $(0, 0, 1)$ is one of the two generators of $Z(\Lambda_r)$. Therefore, $\Lambda_r$ is isomorphic to $\Gamma_r$.

\end{remark}

Let $\Lambda_r$ be as defined in (\ref{eq:lattice}) and let assume that $\Heis_n$, $\Lambda_r \backslash \Heis_n$ are endowed with the standard metric. Since the quotient projection $\pi: \Heis_n \to \Lambda_r \backslash \Heis_n$ is a Riemannian submersion and furthermore a local isometry, we can identify the tangent bundle of $\Lambda \backslash \Heis_n$ with $(\Lambda_r \backslash \Heis_n) \times \hh_n$. The projection $\pi$ maps geodesics into geodesics and the energy function $\tilde E$ of $T(\Lambda_r \backslash \Heis_n)$ is related to the energy function $E$ of $T\Heis_n$ by
$$
\tilde E(\Lambda_rp,Y) = E(p, Y) = \frac{1}{2}\langle Y, Y\rangle
$$
and this does not depend on the given representative.

Since the integrals $f_{Z_1}$, $g_{A_i}$ of the geodesic flow of $T\Heis_n$, as given in Theorem \ref{heisenberg-integrable} do not depend on the first coordinate, they descend to first integrals 
$$
\tilde f_{Z_1}(\Lambda_rp,Y) = f_{Z_1}(p, Y), \qquad \tilde g_{A_i}(\Lambda_rp, Y) = g_{A_i}(p, Y)
$$
of the geodesic flow of $T(\Lambda_r \backslash \Heis_n)$. Moreover, such first integrals are in involution, since for all $f,g \in C^\infty(T(\Lambda_r \backslash \Heis_n))$ we have
$$
\{f \circ \pi, g \circ \pi\} = \{f, g\} \circ \pi.
$$

Note that the integrals $F_k$, $k = 1, \ldots, 2n$, from Theorem \ref{heisenberg-integrable} do not descend to the quotient. However, let $(p, Y) \in T\Heis_n$ and $q \in \Lambda_r$, say $p = (x,y,z)$ and $q = (x',y',z')$, with $x,y \in \mathbb R^n$, $z \in \mathbb R$, $x' \in r\mathbb Z^n$, $y' \in 2\mathbb Z^n$ and $z' \in \mathbb Z$. Take $W, W', W''\in  \hh_n$ such that  $\exp W=  p$,  $\exp W'=q$ and $\exp W''=qp$. Observe that $W''_{\vv}= W_{\vv} + W'_{\vv}$. Thus one has  
\begin{align*}
  F_k(qp, Y) 
  & = \langle Y, X_k\rangle - \langle Y, Z_1\rangle\langle J(W'_\mathfrak v +  W_\mathfrak v), X_k\rangle \\
  & = F_k(p, Y) - f_{Z_1}(p, Y)\langle J W'_\mathfrak v, X_k\rangle.
\end{align*}
Since $\langle JW'_\mathfrak v, X_k\rangle \in \mathbb Z$ we have that 
$$
F_k(qp, Y) = F_k(p, Y) \mod f_{Z_1}(p, Y)\mathbb Z
$$
and since $f_{Z_1}$ is a first integral of the geodesic flow we have that the function
$$
\hat F_k(p, Y) = \sin\left(2\pi\frac{F_k(p, Y)}{f_{Z_1}(p, Y)}\right)
$$
descends to $\Lambda_r \backslash \Heis_n$ and is constant along the integral curves of the geodesic vector field in $T(\Lambda_r \backslash \Heis_n)$. In order to get a smooth first integral let 
$$
\bar F_k(p, Y) = e^{-1/f_{Z_1}(p, Y)^2}\hat F_k(p, Y)
$$
and let us define
$$
\tilde F_k(\Lambda_rp, Y) = \bar F_k(p, Y).
$$

So the functions $\tilde F_k$ are smooth (non-analytic) first integrals for the geodesic flow on $T(\Lambda_r \backslash \Heis_n)$. It follows from a direct calculation that the families $f_{Z_1}, g_{A_i}, \bar F_{2k - 1}$ and $f_{Z_1}, g_{A_i}, \bar F_{2k}$, $i, k = 1, \ldots, n$ are in involution. So the geodesic flow in $T(\Lambda_r \backslash \Heis_n)$ is completely integrable in the sense of Liouville. 

\begin{cor}
Let $\Heis_n$ be the Heisenberg Lie group endowed with the standard metric, and let $\Lambda_r$ defined as in (\ref{eq:lattice}). If $\Lambda_r \backslash \Heis_n$ is the corresponding Heisenberg manifold, then the geodesic flow in $T(\Lambda_r \backslash \Heis_n)$ is completely integrable with smooth first integrals.
\end{cor}

\section{The case of a general left-invariant metric on $\Heis_n$}

In this section we show how to construct first integrals for the geodesic flow on $T\Heis_n$ with arbitrary left-invariant metrics on $\Heis_n$. Recall that any left invariant metric $g$ on $\Heis_n$ is isometric to one of the form $\langle\cdot, \cdot\rangle_P$ defined as follows. Let $\langle\cdot, \cdot\rangle$ denote the standard inner product on $\hh_n$ and let $P$ be a symmetric positive definite operator on $\hh_n$, with respect to the standard inner product, which has the matrix form
$$
P = 
\begin{pmatrix}
  \tilde P & \\
   & \lambda
\end{pmatrix},
$$
where $\tilde P: \mathfrak v \to \mathfrak v$ is symmetric and positive definite and $\lambda > 0$. We can think of $\tilde P$ as a symmetric matrix with positive eigenvalues. The metric $\langle\cdot, \cdot\rangle_P$ is defined in $\hh_n$ by
$$
\langle X, Y\rangle_P = \langle PX, Y\rangle.
$$

Let $j_P(Z)$ be the $\langle\cdot, \cdot\rangle_P$-skew-symmetric operator such that 
$$
\langle j_P(Z)X, Y\rangle_P = \langle [X, Y], Z\rangle_P.
$$

Note that
\begin{align*}
  \langle [X, Y], Z\rangle_P & = \langle [X, Y], PZ\rangle =  \langle [X, Y], \lambda Z\rangle \\
  & = \lambda \langle j(Z)X, Y\rangle = \lambda\langle PP^{-1}j(Z)X, Y\rangle \\
  & = \lambda\langle P^{-1}j(Z)X, Y\rangle_P.
\end{align*}

This proves the following lemma.

\begin{lem}
  With the assumption and notation of this section, we have that
 $\lambda j(Z) = \tilde P j_P(Z)$, or equivalently, $j_P(Z) = \lambda \tilde P^{-1}j(Z)$, for all $Z \in \mathfrak z$.
\end{lem}

We also need the next result.

\begin{lem}
If $f: T\Heis_n \to \mathbb R$ and $\grad_{(p, Y)}f$, $\grad_{(p, Y)}^P f$ are gradients of $f$ with respect to the metrics on $T\Heis_n$ induced by the standard metric and $\langle\cdot,\cdot\rangle_P$ respectively, then
  $$
  \grad_{(p, Y)}f = P\grad_{(p, Y)}^P f,
  $$
where $P$ acts diagonally on $T_{(p, Y)}(T\Heis_n)$. 
\end{lem}

\begin{proof}
By a direct calculation,
$$
\langle\grad_{(p, Y)}^P f, S\rangle_P = \langle P\grad_{(p, Y)}^P f, S\rangle = df_{(p, Y)}(S) = \langle \grad_{(p, Y)} f, S\rangle
$$
which proves the lemma.
\end{proof}

So, if $\grad_{(p, Y)} f = (U, V)$, then
$$
\grad_{(p, Y)}^P f = (P^{-1}U, P^{-1}V).
$$

It follows from Proposition \ref{propcuad} that $f_{Z_1}$ is also a first integral for the geodesic flow on $(T\Heis_n, \langle\cdot, \cdot\rangle_P$). 

In order to find quadratic first integrals, we use the following lemma.

\begin{lem}
  An endomorphism $A$ of $\mathfrak v$ is symmetric with respect to $\langle\cdot,\cdot\rangle_P$ if and only if $PA = A^tP$, where $A^t$ is the transpose of $A$ with respect to the standard metric.
\end{lem}

\begin{proof}
  Let $X, Y \in \mathfrak h_n$ be arbitrary. Then $\langle AX, Y\rangle_P = \langle X, AY\rangle_P$ if and only if $\langle PAX, Y\rangle = \langle PX, AY\rangle$, and this holds if and only if $\langle PAX, Y\rangle = \langle A^tPX, Y\rangle$.
\end{proof}

 Since $\tilde P$ is symmetric with respect to the standard metric, there exist a basis $U_1, \ldots, U_{2n}$ of $\mathfrak v$ such that $\tilde PU_i = \lambda_iU_i$ (with $\lambda_i > 0$, since $\tilde P$ is positive definite). This basis can be chosen orthonormal with respect to the standard metric, and moreover, since $j(Z)$ acts on $\mathfrak v$ as a multiple of the multiplication by $\sqrt{-1}$ on $\mathbb C^n$, we can assume that $U_{2i} = j(Z_1)U_{2i - 1}$. So we can define the operators $\tilde A_i$ on $\mathfrak v$ such that 
$$\tilde A_i U_{2i - 1} = U_{2i - 1}\quad \tilde A_iU_{2i} = U_{2i} \mbox{ and  } \tilde A_iU_k = 0 \,\, \mbox{ if } k \neq 2i - 1, 2i.$$
  
It follows that 
\begin{itemize}
\item $\tilde A_i$ is symmetric with respect to $\langle\cdot, \cdot\rangle_P$, 
 \item $\tilde A_i$ commutes with $j_P(Z)$ and 
\item $[\tilde A_i, \tilde A_j] = 0$ for all $i, j$. 
\end{itemize}
So Theorem \ref{teo1} shows that we get a commuting independent family of first integrals for the geodesic flow on $(T\Heis_n, \langle\cdot, \cdot\rangle_P)$. Namely
$$
g_{\tilde A_i}(p, Y) =  \frac{1}{2}\langle \tilde A_iY, Y\rangle_P,
$$
where $\tilde A_i$ is extended so that $\tilde A_iZ_1 = 0$.

Finally, in order to obtain the remaining integrals we use the Killing fields $U_k^*$, $k = 1, \ldots, 2n$. In fact, we have, in the same manner as in the previous section, that
$$
\tilde F_k(p, Y) = \langle Y, U_k\rangle_P - \langle j_P(Y_{\mathfrak z})W_{\mathfrak v}, U_k\rangle_P
$$
forms a family of first integrals for the geodesic flow such that 
$$\{\tilde F_{2i - 1}, \tilde F_{2j}\}^P =\delta_{ij}f_{Z_1}.$$ 

\begin{theorem}\label{heisenberg-integrable-2}
   The geodesic flow on $(T\Heis_n, \langle\cdot, \cdot\rangle_P)$ is completely integrable in the sense of Liouville. Moreover the sets
  \begin{enumerate}
  \item 
    $$
    \mathcal F = \{f_{Z_1}\} \cup \{g_{\tilde A_i}\}_{i = 1}^n \cup \{\tilde F_{2k - 1}\}_{k = 1}^n
    $$
  \item  and
    $$
    \mathcal F' = \{f_{Z_1}\} \cup \{g_{\tilde A_i}\}_{i = 1}^n \cup \{\tilde F_{2k}\}_{k = 1}^n
    $$
     \end{enumerate}
    give two independent commuting families of first integrals of the geodesic flow.
\end{theorem}

\begin{remark} Note that the case of a general Riemannian left-invariant metric on $\Heis_n$ does not give a naturally reductive space. Also the compact quotients considered here  are not globally homogeneous  but locally. 
\end{remark}


\end{document}